\documentclass[11pt,a4paper]{amsart}
\usepackage[utf8]{inputenc}
\usepackage{amsmath}
\usepackage{amsfonts}
\usepackage{amssymb}
\usepackage{graphicx}
\usepackage{blindtext}
\usepackage{amsmath,amssymb, amsthm}
\usepackage{pgf}
\usepackage{tikz-cd}
\usepackage{xy}
\usepackage{hyperref}
\usetikzlibrary{matrix,arrows,decorations.pathmorphing}
\usepackage[english]{babel}
\usepackage{mathtools}
\usepackage{caption}
\usepackage{geometry}
\usepackage{float}
\usepackage{diagbox}

\theoremstyle{definition}

\theoremstyle{definition}
\newtheorem{rmk}{Remark}
\theoremstyle{definition}

\theoremstyle{definition}\usepackage{amsmath}

\theoremstyle{plain}
\newtheorem{thm}{Theorem}
\theoremstyle{plain}

\theoremstyle{plain}
\newtheorem{conj}{Conjecture}
\theoremstyle{plain}
\newtheorem{lem}{Lemma}

\newtheorem{coro}{Corollary}
\newenvironment{aprf}{\noindent\textbf{(Almost) proof:\\}}{{ $\square$ \par}}

\newcommand\blfootnote[1]{%
  \begingroup
  \renewcommand\thefootnote{}\footnote{#1}%
  \addtocounter{footnote}{-1}%
  \endgroup
}

\newcommand{\Q}{\mathbb{Q}}
\newcommand{\R}{\mathbb{R}}
\newcommand{\C}{\mathbb{C}}

\newcommand{\Ok}{\mathcal{O}_K}

\newcommand{\epsi}{\varepsilon}

\newcommand{\Nm}{\operatorname{N}}
\newcommand{\diff}{\mathcal{D}_{K}}
\newcommand{\lp}{\left(}
\newcommand{\rp}{\right)}
\newcommand{\N}{\mathbb{N}}

\newcommand{\Fr}{\mathcal{F}_{r_1,r_2}}

\begin{document}

\address{Dipartimento di Matematica\\
         Università di Milano\\
         via Saldini 50\\
         20133 Milano\\
         Italy}
\title[Conjectural improvements for regulators]{A conjectural improvement for inequalities related to regulators of number fields}
\author[F.~Battistoni]{Francesco Battistoni}

\email{francesco.battistoni@unimi.it}



\blfootnote{The author states that there is no conflict of interest.}

\maketitle

\begin{abstract}
An inequality proved firstly by Remak and then generalized by Friedman shows that
there are only finitely many number fields with a fixed signature and whose regulator is less
than a prescribed bound. Using this inequality, Astudillo, Diaz y Diaz, Friedman and
Ramirez-Raposo succeeded to detect all fields with small regulators having degree less or equal
than 7.

In this paper we show that a certain upper bound for a suitable polynomial,
if true, can improve Remak-Friedman's inequality and allows a classification
for some signatures in degree 8 and better results in degree 5 and 7.
The validity of the conjectured upper bound is extensively discussed.
\end{abstract}

\tableofcontents

\section{Introduction}
Let $K$ be a number field of degree $n$. Let $d_K$ be the discriminant of $K$ and $R_K$ be the regulator. Landau \cite{remak} proved that there exists $A_1>0$ such that
$$R_K < A_1\frac{\sqrt{|d_K|}}{(\log|d_K|)^{n-1}}$$
while Remak \cite{remak} showed that there exists $A_2>0$ depending on the signature of $K$ such that 
$$|d_K| < n^n\exp(A_2 R_K).$$
In particular, Remak's inequality shows that an upper bound for $R_K$ implies an upper bound for $d_K$, and this gives a finite number of fields $K$ satisfying this bound by Hermite's theorem.\\
Let $K$ be of degree $n$, admitting $r_1$ embeddings in $\R$ and $r_2$ embeddings up to complex conjugation in $\C$ which are not embeddings in $\R$: the pair $(r_1,r_2)$ is called signature of $K$ and satisfies the relation $n=r_1+2r_2$. Friedman \cite{friedmanAnalyticRegulator, friedmanTotalPositivity} generalized Remak's inequality to every number field extension $K/F$, showing that there exists $A_3>0$ depending on the signatures of $K$ and $F$ such that
$$ |d_K| < |d_F|^{[K:F]} [K:F]^{[K:\Q]}\exp(A_3 R_K).$$
This expression reduces to the previous one assuming $F=\Q$, and it will be referred to as ``Remak-Friedman's inequality".
 Starting from this result, a classification procedure was set in 2016 by Astudillo, Diaz y Diaz and Friedman \cite{regulators} which allowed to detect all number fields with regulator less than some chosen upper bound (depending on the signature) in the following cases.
\begin{itemize}
    \item For degrees $n\leq 6$.
    \item For every signature in the degree $n=7$, except $(5,1)$.
    \item For signatures $(0,4)$ and $(8,0)$ in degree $n=8$ and for the signature $(9,0)$ in degree $n=9$.
\end{itemize}
The missing signature $(5,1)$ in degree 7 was discussed later by Friedman and Ramirez-Raposo \cite{RamirezRaposo} with an ad hoc improvement to Remak-Friedman's inequality which allowed to implement the procedure.

The next cases in which a classification of this kind is not known yet are the remaining
signatures (2,3), (4,2) and (6,1) in degree 8. One of the reasons why this study was skipped by the previous authors was the lack of complete tables of number fields up to some discriminant bounds, which prove to be crucial to guarantee the correctness of the procedure. We were able to provide such lists for these signatures \cite{battistoniMinimum, battistoni2019Arxiv}, and we tried to apply the classification method on the considered signatures; the attempt however was not successful, and the reasons are similar to the ones which prevented the previous authors to immediately solve the case of signature (5,1).

It is then natural to try to overcome this difficulty by looking at what Friedman and Ramirez-Raposo did for signature (5,1), and their work suggests that a possible solution could derive from improving somehow the upper bounds of Remak-Friedman's inequality.

This paper is mainly devoted to the study of a specific term of Remak-Friedman's inequality, corresponding to a multivariate polynomial $P$ defined over a subregion of the hypercube $[-1,1]^{n}$ which depends on the signature $(r_1,r_2)$: more in detail, we would like to provide the correct optimization of $P$ and detect its maximum value, in order to apply it for the classification of number fields with small regulators. The setting of the problem takes inspiration from Pohst \cite{pohstRegulator}, who proved that this maximum is indeed much lower than the usual estimate for $P$ whenever all the variables of the polynomial are real and the degree of the involved fields is at most 11 (which is the reason why signatures (8,0) and (9,0) were solved). At the same time, it is easy to show that, if the variables are all complex, the previously known maximum for $P$ is sharp.

We would like then to obtain similar results for intermediate signatures. Apart from a few exceptions, we were not able to analytically provide the correct values of these maximums: however, we made numerical experiments, via programs written for the computer algebra system PARI/GP \cite{pari} and using the MATLAB Optimization Toolbox \cite{matlabOptimization}, and from the heuristics we conjectured the values of the desired maximums and an iterative behaviour of these values.

Finally, assuming the truth of these conjectures as given in Table \ref{TableMaximum} of Section \ref{section5}, we prove the following results (the involved quantities $M(n,r_2)$ are defined in \eqref{definitionM} of Section \ref{subsection4.1}). 

\begin{thm}\label{ThmDeg8}
Suppose the value of $M(8,1)$ given in Table \ref{TableMaximum} is correct. Then there exist exactly 4 number fields $K$ of signature $(6,1)$ with regulator $R_K\leq 7.48$, and they are the 4 fields with this signature and $|d_K|=65106259, 
68494627, 
68856875, 
69367411$, having $R_K=7.13506\ldots,7.38088\ldots,7.41473\ldots,7.4303\ldots$ respectively.
\end{thm}

\begin{thm}\label{ThmDeg5}
Suppose the value of $M(5,1)$ given in Table \ref{TableMaximum} is correct. Then any number field with signature $(3,1)$ and $|d_K|> 48000$ must have $R_K> 2.15$; among the 145 fields of this signature with $|d_K|\leq 48000$ there exist exactly 40 fields with $R_K\leq 2.15$, and they satisfy $|d_K|\leq 25679$. 
\end{thm}

\begin{thm}\label{ThmDeg7}
Suppose the value of $M(7,1)$ given in Table \ref{TableMaximum} is correct. Then any number field with signature $(5,1)$ and $|d_K|> 2\cdot 10^7$ must have $R_K> 8$; among the 528 fields of this signature with $|d_K|\leq 2\cdot 10^7$ there exist exactly 135 fields with $R_K\leq 8$, and they satisfy $|d_K|\leq 11755159$.  
\end{thm}
\noindent
This last result is a conjectural improvement to the classification of number fields with signature $(5,1)$ given in \cite{RamirezRaposo}. The fields detected in these theorems are completely and explicitly described in the tables of number fields with the corresponding signatures presented in Kl\"{u}ners-Malle Database \cite{klunersMalle} and LMFDB database \cite{lmfdb}.



\subsection*{Acknowledgments} I would like to thank my Ph.D. advisor Giuseppe Molteni for having introduced me to the problem and for his supervision of my work, and Eduardo Friedman for giving me insights on the current state of the research on minimum regulators. I also thank the anonymous referee for all their remarks and suggestions which helped me in improving this paper.

\section{Classifying number fields with small regulator}

\subsection{Remak-Friedman's estimate}
Let $\Fr$ be the family of number fields with signature $(r_1,r_2)$, and let $K\in\Fr$. Let $\Ok$ be its ring of integers, $\Ok^*$ the subgroup of units, $\mu_K$ the finite subgroup of roots of unity contained in $K$ and $\infty_K$ the set of archimedean places of $K$: for every $\epsi\in\Ok^*$ define 
$$m_K(\epsi)\coloneqq \lp\sum_{i=1}^{r_1+r_2} (\log ||\epsi||_{v_i})^2\rp^{1/2}$$
where $||\epsi||_{v_i}$ is the absolute value corresponding to the $i$-th archimedean place $v_i\in\infty_K$. One has $m_K(\epsi)=0$ if and only if $\epsi\in\mu_K$ by Kronecker's theorem.
\begin{lem}\label{LemmaLogReg}
    Let $K\in\Fr$, and let $r\coloneqq r_1+r_2-1>0$. There exists $\epsi\in\Ok^*\setminus \mu_K$ such that
    \begin{equation}\label{EstimateLogReg}
        m_K(\epsi) \leq \gamma_r^{1/2} \lp\sqrt{r+1}R_K\rp^{1/r},
    \end{equation}
    where $\gamma_r$ is the Hermite constant of dimension $r$.
\end{lem}

\begin{proof}
The quotient group $\Ok^*/\mu_K$ can be embedded as a lattice of dimension $r$ in the hyperplane $\{x_1+\cdots+x_{r+1}=0\}\subset\R^{r+1}$: the embedding is given by the map which sends $\epsi\in \Ok^*/\mu_K$ to $(\log ||\epsi||_{v_i})_{i=1}^{r+1}$, and in this space $m_K(\epsi)$ becomes the Euclidean length of $\epsi$. 

The definition of $R_K$ is such that the covolume of the lattice is equal to $\sqrt{r+1}R_K$. The claim follows then from Minkowski's theorem on successive minima \cite[ pp. 120, 205, 332]{casselsGeometryNumbers}. For a definition of Hermite constants see \cite[ Chapter 3, Section 3]{pohst1997algorithmic}.
\end{proof}


\begin{thm}[Remak, Friedman]
Let $K\in \Fr$ and assume $K=F(\epsi)$ with $F$ a number field and $\epsi\in\Ok^*\setminus\mu_K$. Then 
\begin{equation}\label{EstimateRemakFriedman}
    \log |d_K|\leq [K:F]\log |d_F| + [K:\Q]\log[K:F] + m_K(\epsi)A(K/F) 
\end{equation}
where 
$$A(K/F)\coloneqq \sqrt{\frac{1}{3}\sum_{v\in\infty_F} ([K:F]^3-[K:F]-4r_2(v)^3-2r_2(v)) }$$
and $r_2(v)$ is defined as the number of complex places of $K$ which extend $v$ whenever $v$ is real, otherwise it is zero.
\end{thm}

\begin{proof}
See \cite[Lemma 3.4]{friedmanAnalyticRegulator}.
\end{proof}
\noindent
Assume $K=\Q(\epsi)$ is of degree $n$, with $\epsi$ as in Lemma \ref{LemmaLogReg}: then inequalities \eqref{EstimateLogReg} and \eqref{EstimateRemakFriedman} give 
\begin{equation}\label{D1}
\log |d_K|\leq n\log n + \gamma_r^{1/2} \lp\sqrt{r+1}R_K\rp^{1/r}A(K/\Q)=: D_1(R_K,n,r_2).     
\end{equation}

If instead $\Q(\epsi)\subsetneq K,$ then an iterative adaptation of  \eqref{EstimateRemakFriedman}, obtained by applying the former result to consecutive pair of fields in the tower $\Q\subset \Q(\epsi)=F_1\subset F_2\subset\cdots\subset F_m=K$, gives a different upper bound $\log|d_K|\leq D_2(R_K,n,r_2).$  In any case, we obtain an estimate 
\begin{equation}\label{GeneralizedRemakFriedman}
    \log|d_K|\leq D(R_K,n,r_2) \hspace{0.5cm}
\end{equation}
with $D$ is equal to either $D_1$ or $D_2$ depending on $K=\Q(\epsi)$ or not.

\begin{rmk}
If the degree $n$ of $K$ is a prime number, then $K=\Q(\epsi)$ and we immediately have $D=D_1$.
\end{rmk}

\subsection{Analytic lower bounds}
The previous number-geometric estimates give an upper bound for the discriminant in terms of the regulator. We recall now an estimate which gives instead a lower bound for the regulator and is crucial for the success of the classification of number fields with small regulators.

For every signature $(r_1,r_2)$ let $n\coloneqq r_1+2r_2$. Let $g_{r_1,r_2}:(0,+\infty)\rightarrow \R$ be the function defined as
\begin{equation*}
    g_{r_1,r_2}(x)\coloneqq \frac{1}{2^{r_1}4\pi i}\int_{2-i\infty}^{2+i\infty}\lp\pi^n 4^{r_2}\cdot x\rp^{-s/2}\Gamma(s/2)^{r_1}\Gamma(s)^{r_2}(2s-1)ds
\end{equation*}
where $\Gamma(s)$ is the usual Gamma function.
\begin{lem}\label{lemmapropg}
The function $g_{r_1,r_2}$ satisfies the following properties.
\begin{itemize}
    \item $g_{r_1,r_2}$ is a smooth function over $(0,+\infty)$. 
     \item $g_{r_1,r_2}$ has a unique zero $x_0\in (0,+\infty)$, $g_{r_1,r_2}(x)<0$ for $x<x_0$ and $g_{r_1,r_2}(x)>0$ for $x> x_0$.
    \item $g_{r_1,r_2}$ has a unique critical point $x_1\in (0,+\infty)$, which is a maximum point. 
    \item For every open interval $(a,b)\subset (0,+\infty)$ and for every $x\in(a,b)$ one has $g_{r_1,r_2}(x) > min (g_{r_1,r_2}(a),g_{r_1,r_2}(b))$.
\end{itemize}
\end{lem}
\begin{proof}
The smoothness of the function $g_{r_1,r_2}$ is a direct consequence of the absolute convergence of the integral
$$\int_{2-i\infty}^{2+\infty} x^{-s}\cdot s\Gamma(s/2)^{n}\Gamma(s)^{m}ds$$
for $x\in [a,b]\subset (0,+\infty)$ and $n,m>0$, where the convergence is a consequence of Stirling's asymptotic formula for the Gamma function.

The second and the third statement are proven in \cite[Theorem 2]{friedmanTotalPositivity} using the theory of total positivity developed by Karlin \cite{karlinTotalPositivity}, and the fourth one is an immediate consequence of these.
\end{proof}
The link between the function $g_{r_1,r_2}$ and the regulators of number fields is explicitly presented in the following formula.
\begin{thm}
Let $K\in\Fr$ with discriminant $d_K$ and regulator $R_K$ and let $w_K\coloneqq \#\mu_K$. Then
\begin{equation}\label{explicitFormula}
    \frac{R_K}{w_K} = \sum_{\mathfrak{a}\in\mathfrak{A}}g_{r_1,r_2}\lp\frac{\Nm(\mathfrak{a})^2}{|d_K|}\rp + \sum_{\mathfrak{b}\in\mathfrak{B}}g_{r_1,r_2}\lp\frac{\Nm(\mathfrak{b})^2}{|d_K|}\rp
\end{equation}
where $\mathfrak{A}$ is the set of principal ideals of $\Ok$, $\mathfrak{B}$ is the set of ideals in $\Ok$ equivalent to $\diff$ in the class group of $K$, with $\diff$ being the different ideal of $\Ok$, and $\Nm(\mathfrak{I})$ is the absolute norm of an ideal $\mathfrak{I}\subset\Ok$.
\end{thm}

\begin{proof}
This is proved in \cite[Theorem 2.1]{friedmanAnalyticRegulator}. The proof relies on the meromorphic function
$$\xi_K(s)\coloneqq |d_K|^{s/2}\lp\pi^{-s/2}\Gamma(s/2)\rp^{r_1}\Big(\lp 2\pi\rp^{-s}\Gamma(s)\Big)^{r_2}\lp\sum_{\mathfrak{a}\in\mathfrak{A}}\Nm(\mathfrak{a})^{-s}+\sum_{\mathfrak{b}\in\mathfrak{B}}\Nm(\mathfrak{b})^{-s}\rp$$
which satisfies the functional equation $\xi_K(1-s)=\xi_K(s)$ and has poles, which are simple, only at 0 and 1, the residue at 1 being equal to $(2^{r_1}R_Kc_K)/w_K$ where $c_K=2$ if the different ideal $\diff$ is not principal and $c_K=1$ otherwise. From this, one obtains the equality
\begin{align*}
    \frac{1}{2\pi i}\int_{2-i\infty}^{2+i\infty}(2s-1)\xi_K(s)ds = \frac{2^{r_1}R_Kc_K}{w_K} + \frac{1}{2\pi i}\int_{1/2-i\infty}^{1/2+i\infty}(2s-1)\xi_K(s)ds
\end{align*}
where the last integral is 0 since $(2s-1)\xi_K(s)$ is odd with respect to the transformation $s\to 1-s$. The obtained expression is finally verified to be equal to \eqref{explicitFormula}.
\end{proof}

\begin{coro}\label{coroprimolower}
Let $K\in\Fr$ with discriminant $d_K$ and regulator $R_K$, and assume we have\\ $g_{r_1,r_2}(1/|d_K|)>0$. Then
    \begin{equation}\label{LowerBoundRegolatore}
        R_K \geq 2g_{r_1,r_2}\lp\frac{1}{|d_K|}\rp.
    \end{equation}
\end{coro}
\begin{proof}
The trivial ideal in $\Ok$ is principal and with norm 1, giving rise to the term $g_{r_1,r_2}\lp 1/|d_K|\rp$ in the right hand side of \eqref{explicitFormula}. Since this term is positive, every other term in the sum must be positive by the second statement of Lemma \ref{lemmapropg}, so that $R_K/w_K\geq g_{r_1,r_2}(1/|d_K|)$. We obtain \eqref{LowerBoundRegolatore} since $w_K\geq 2$.
\end{proof}

\begin{coro}\label{CoroLowerBoundReg}
Let $K\in\Fr$, and let $R_0>0$. Assume that $0<d_1\leq |d_K|\leq d_2$ and $2g_{r_1,r_2}(1/d_1)>R_0$, $2g_{r_1,r_2}(1/d_2)>R_0$. Then 
$$R_K \geq 2g_{r_1,r_2}(1/|d_K|) > R_0.$$
\end{coro}

\begin{proof}
This is a direct consequence of the fourth statement of Lemma \ref{lemmapropg} and of the previous corollary.
\end{proof}

\begin{rmk}
Given $K\in\Fr$, if the different ideal $\diff$ is principal then the two sums in the right hand side of \eqref{explicitFormula} are equal: this means that the factor 2 in the corollaries \ref{coroprimolower} and \ref{CoroLowerBoundReg} can be replaced with a factor 4. 

There is a sufficient criterion to determine if $\diff$ is principal. We need a definition: if $K$ is a number field of degree $n$ and discriminant $d_K$, the root discriminant of $K$ is $|d_K|^{1/n}$. It is known that for every signature $(r_1,r_2)$ there exists $\delta_{r_1,r_2}>0$ such that $|d_K|^{1/n} > \delta_{r_1,r_2}$ for every $K\in\mathcal{F}_{r_1,r_2}$.

\begin{lem}\label{lemdifferent}
Let $K\in\Fr$ and assume $|d_K|^{1/[K:\Q]}< \delta_{3r_1,3r_2}$ where $\delta_{3r_1,3r_2}$ is a proved lower bound for the root discriminant of fields with signature $(3r_1,3r_2)$. Then $\diff\subset \Ok$ is a princiapl ideal.
\end{lem}

\begin{proof}
Assume $\diff$ is not principal. Since the class of $\diff$ is known to be a square in the class group of $K$ \cite[p.234]{heckeLectures}, the size of this group must be at least 3. By Class Field Theory, there exists an abelian unramified extension $L/K$ of degree at least 3, and the absence of ramification implies
$$|d_K|^{1/[K:\Q]} = |d_L|^{1/[L:\Q]},\hspace{0.2cm} \frac{r_1(K)}{[K:\Q]}=\frac{r_1(L)}{[L:\Q]}.$$
Now, if the ratio $r_1/n$ is fixed, the root discriminant increases monotonically in the degree $n$ of the field \cite{pohst1975berechnung} and so $|d_K|^{1/[K:\Q]}\geq \delta_{3r_1,3r_2}$. 
\end{proof}
Lower bounds for root discriminants can be found in Diaz y Diaz' tables \cite{y1980tables} for several degrees and signatures.
\end{rmk}

\subsection{The procedure} We illustrate the method proposed by Astudillo, Diaz y Diaz and Friedman \cite{regulators} for the classification of fields with given signature and bounded regulator: it is based upon Remak-Friedman's inequality \eqref{EstimateRemakFriedman}, Corollary \ref{CoroLowerBoundReg} and Lemma \ref{lemdifferent}, and it allowed the authors to get the classification of number fields with small regulators for the degrees and signatures described in the introduction.
\begin{itemize}
    \item[a)] Choose an upper bound $R_0$ for $R_K$, and compute $D(R_0,n,r_2)$ as in \eqref{GeneralizedRemakFriedman} using Remak-Friedman's inequality. Then, a field $K\in\Fr$ with $R_K\leq R_0$ is forced to have $|d_K|\leq \exp(D(R_0,n,r_2)).$
    \vspace{0.2cm}
    \item[b)] If $\exp(D(R_0,n,r_2))< \delta_{3r_1,3r_2}$, verify that $4g_{r_1,r_2}(\exp(-D(R_0,n,r_2)))> R_0$ and look for the smallest possible value $d_1>0$ such that $4g_{r_1,r_2}(1/d_1)>R_0$. Then a field $K\in\Fr$ with $R_K\leq R_0$ must have $|d_K|\leq d_1$.\vspace{0.2cm}
    \item[c)] If $\exp(D(R_0,n,r_2))\geq \delta_{3r_1,3r_2}$, verify that $2g_{r_1,r_2}(\exp(-D(R_0,n,r_2)))> R_0$ and look for the smallest possible value $d_2>0$ such that $2g_{r_1,r_2}(1/d_2)>R_0$. Then repeat b) by replacing $\exp(D(R_0,n,r_2))$ with $d_2$.\vspace{0.2cm}
    \item[d)] Given $d_1$ from b) or c), compute the regulator of every number field $K\in\Fr$ with $|d_K|<d_1$ and list the fields with $R_K\leq R_0$.
\end{itemize}
    There are two important conditions which must be satisfied in order for this procedure to work properly.
    \begin{itemize}
        \item[1)] One needs  $4g_{r_1,r_2}(\exp(-D(R_0,n,r_2)))$ (or $2g_{r_1,r_2}(\exp(-D(R_0,n,r_2)))$) to be larger than $R_0$. If these values are slightly below $R_0$, a refinement of the procedure involving further terms in the right hand side of \eqref{explicitFormula} may be helpful (see \cite[Lemma 5]{regulators}).\vspace{0.15cm}
        \item[2)]  One needs complete tables of number fields $K\in\Fr$ with $|d_K|\leq d_1$, so that one can be certain to have considered every field of $\Fr$.
    \end{itemize}

\begin{rmk}\label{RemarkGRH}
    The computation of $R_K$ for $|d_K|\leq d_1$ is done via PARI/GP, which provides a value of $R_K$ whose correctness depends on the assumption of Generalized Riemann Hypothesis for Dedekind Zeta functions \cite[p. 353]{cohenComputational}. However, Astudillo, Diaz y Diaz and Friedman \cite[Section 3.3]{regulators}  give a condition to verify that the output value $\Tilde{R}_K$ is unconditionally correct: in fact, the output $\Tilde{R}_K$ is an integral multiple of the true value $R_K$, since the algorithm unconditionally identifies a subgroup of $\Ok^*$ of finite index. If $\Tilde{R}_K = m R_K$ with $m\in\N$, looking back at the explicit formula \eqref{explicitFormula} and assuming $g_{r_1,r_2}(1/|d_K|)>0$, we obtain
    \begin{equation*}
         0<m=\frac{\Tilde{R}_K}{R_K} <\frac{\Tilde{R}_K}{4\cdot\lp g_{r_1,r_2}\lp\frac{1}{|d_K|}\rp + c\cdot g_{r_1,r_2}\lp\frac{4}{|d_K|}\rp \rp}
    \end{equation*}
  where $c$ is the number of ideals in $\Ok$ with norm 2. In order to have $\Tilde{R}_K = R_K$ it is thus sufficient to have  
     \begin{equation}\label{CorrettoRegolatoreCaso2}
       g_{r_1,r_2}\lp\frac{1}{|d_K|}\rp>0,\hspace{0.2cm} 0<\frac{\Tilde{R}_K}{4\cdot\lp g_{r_1,r_2}\lp\frac{1}{|d_K|}\rp + c\cdot g_{r_1,r_2}\lp\frac{4}{|d_K|}\rp \rp}<2.
    \end{equation}

\end{rmk}

\subsection{Attempt of classification for signatures in degree 8}\label{sectionfirstattempt}
Let us apply the previous procedure to signatures (2,3), (4,2) and (6,1) in degree 8, for which no classification has been given already.

Let $(r_1,r_2)$ be one of these signatures. From \cite{battistoniMinimum, battistoni2019Arxiv} one gets complete tables of fields $K\in\Fr$ with $|d_K|\leq d_1$, where $d_1$ depends on the signature and is specified in Table \ref{TabellaSegnature8}. For every field in the list, we compute their regulators $R_K$ using PARI/GP and we verify that the output satisfies \eqref{CorrettoRegolatoreCaso2}. We define $R_m$ as the smallest of the detected regulators: one notices that, for each of the considered signatures, $R_m$ is actually attained at the fields with minimum discriminant. Then we choose a number $R_0 >R_m$ obtained from a truncation of the decimal part of $R_m$ and we verify that $4g_{r_1,r_2}(1/d_1)> R_0.$

Next, we use Remak-Friedman's inequality in order to get an upper bound $\log|d_K|\leq D(R_0,n,r_2)$ for fields $K\in\Fr$ with $R_K\leq R_0$: explicit computations show that the bigger upper bound among $D_1(R_0,n,r_2)$ and $D_2(R_0,n,r_2)$ is the first one, given in \eqref{D1}. Furthermore, one verifies that $\exp(D_1(R_0,n,r_2))< \delta_{3r_1,3r_2}$ for the considered signatures and so, if one was able to prove that $4g_{r_1,r_2}(\exp(-D_1(R_0,n,r_2)))> R_0$, a field $K\in\Fr$ with $R_K\leq R_0$ should have $|d_K|\leq d_1$ and would be contained in our lists. 

\begin{table}[H]
\caption{}
\begin{center}
\begin{tabular}{c|c|c|c|c|c|c}
	$(r_1,r_2)$ & $d_1$ & $R_m$ & $R_0$ & $4g(1/d_1)$ & $D_1(R_0,n,r_2)$ & $4g(\exp(-D_1))$\\
	\hline
	(2,3) & 5726300 & 0.83140\ldots & 0.832 & 0.93299\ldots & 32.47101\ldots & -41.14908\ldots\\
	\hline
	(4,2) & 20829049 & 2.29779\ldots & 2.298 & 2.60466\ldots & 38.3603\ldots& -677.394\ldots\\
	\hline
	(6,1) & 79259702 &7.13506\ldots & 7.14 & 7.487499\ldots & 43.7697\ldots & -6926.41\ldots\\
	\hline
\end{tabular}
\end{center}
\label{TabellaSegnature8}
\end{table}

However, as reported in Table \ref{TabellaSegnature8}, the values of $4g_{r_1,r_2}(\exp(-D_1))$ are not only negative, but also large in modulus, and this fact prevents from any improvement possibly obtained by using the adaptation proposed in \cite[Lemma 5]{regulators}. We are thus losing Condition 1) of the procedure, preventing us from applying
completely the classification method for these signatures. 

Given this, the goal of the next section will be to provide insights on how to obtain (conjectural) improvements that allow to solve these difficulties for the examined signatures in degree 8.

\section{Considerations on Remak-Friedman's inequality}
As mentioned in the introduction, condition 1) of the procedure was not satisfied in \cite{regulators} for signature (5,1) already, and the classification of number fields with this signature was accomplished later by Friedman and Ramirez-Raposo \cite{RamirezRaposo}. Their attempt was successful thanks to a modification of  a specific term of Remak-Friedman's inequality: the important contribution consisted in showing that, in order to get the improvement, it was necessary to keep into account the signature of the problem. Inspired by this fact, we look at how Remak-Friedman's inequality is proved and what terms should be considered for a possible improvement.

Remember the assumption on the number field $K$, i.e. $K=\Q(\epsi)$ with $\epsi\in\Ok^*\setminus\mu_K$. Let  $\epsi_1,\epsi_2\ldots,\epsi_n$ be the $K$-conjugates of $\epsi$, ordered so that $|\epsi_i|\leq |\epsi_j|$ for $i\leq j$. Let $D(\epsi)\coloneqq \prod_{1\leq i \leq j\leq n}(\epsi_i -\epsi_j)^2$, then $d_K$ divides $D(\epsi)$ and
\begin{equation}\label{InequalityIntermediate}
    \log |d_K|\leq \log |D(\epsi)| = \log \lp\prod_{1\leq i<j\leq n}\left|1-\frac{\epsi_i}{\epsi_j}\right|^2\rp + \sum_{j=2}^{n}2(j-1)\log |\epsi_j|.
\end{equation}
The second term in the right hand side of \eqref{InequalityIntermediate} is estimated by $m_K(\epsi)\cdot A(K/\Q)$, while the first term is estimated by $n\log n$. Looking at the proof in \cite[Lemma 3.4]{friedmanAnalyticRegulator}, the factor $A(K/\Q)$ is easily seen to be sharp for even degrees, while $m_K(\epsi)$ depends too much on the specific field to provide an inequality better than \eqref{LemmaLogReg}: we look then for an improvement to the other term, and we try this by adopting a more general point of view.

Let $\epsi_1,\ldots,\epsi_n\in\C\setminus\{0\}$ be such that $|\epsi_i|\leq |\epsi_j|$ for $i\leq j$. Define the function
\begin{align}\label{FakeDiscriminant}
    P(\epsi_1,\ldots,\epsi_n)\coloneqq \prod_{1\leq i<j\leq n}\left|1-\frac{\epsi_i}{\epsi_j}\right|^2.
\end{align}
Observe that we can always assume that the numbers $\epsi_j$ are less or equal than 1 in absolute value, because the values of the function $P$ do not change whenever the numbers $\epsi_i$ are divided by $|\epsi_n|$. Thus, we can always think of $P$ as a function defined on the set 
\begin{equation}\label{SetConsidered}
    \{0 < |\epsi_1|\leq |\epsi_2|\leq \cdots\leq |\epsi_n|\leq 1\}.
\end{equation}
The following theorem \cite{bertin} states the basic bound for $P$, which is the one used in Remak-Friedman's inequality.

\begin{thm}\label{ThmBertin}
Let $P(\epsi_1,\ldots,\epsi_n)$ be defined as in (\ref{FakeDiscriminant}). Then $|P(\epsi_1,\ldots,\epsi_n)|\leq n^n$.
\end{thm}


This general estimate does not take into account anything related to the signature. Assume however that the $\epsi_i$'s are all real, i.e. that $K$ is a totally real field.

\begin{thm}\label{thmPohst}
Let $n\in\N$ with $n\geq 2$ and let $\epsi_1,\ldots,\epsi_n$ be real numbers in $[-1,1]$ such that $|\epsi_i|\leq |\epsi_j|$ if $i\leq j$. Then $$P(\epsi_1,\ldots,\epsi_n)\leq 4^{\lfloor\frac{n}{2}\rfloor}.$$
\end{thm}

\begin{proof}
Consider the change of variables $\rho_i\coloneqq \epsi_i/\epsi_{i+1}$ for $i=1,\ldots,n-1$ and rewrite $P(\epsi_1,\ldots,\epsi_n)$ as $P(\rho_1,\ldots,\rho_{n-1})$. Define then $$Q(\rho_1,\ldots,\rho_{n-1})\coloneqq\sqrt{P(\rho_1,\ldots,\rho_{n-1})}$$
which is still a positive function. We look for an estimate of $Q$, which has a simpler form than $P$, over the hypercube $[-1,1]^{n-1}$.\\
Let us analyze some cases in low dimension:
\\\\
$n=2$: the function $Q$ is simply
$$Q(\rho_1) = (1-\rho_1)$$
which is clearly less or equal than 2, this value being attained at $\rho_1=-1$.\\\\
$n=3$: the function $Q$ has now the form
$$\begin{matrix}
    Q(\rho_1,\rho_2) = &(1-\rho_1)&(1-\rho_1\rho_2)\\
    & &(1-\rho_2)
\end{matrix}$$
where the right hand side is assumed to be a product of all the written factors. An easy optimization using the partial derivatives of $Q$ shows that the global maximum is attained on the boundary, precisely at the point $\rho_1=0, \rho_2=-1$ and that the maximum of $Q$ is again equal to 2.
\\\\
$n=4$: the function now assumes the form
$$\begin{matrix}
    Q(\rho_1,\rho_2,\rho_3) = &(1-\rho_1)&(1-\rho_1\rho_2)&(1-\rho_1\rho_2\rho_3)\\
&    &(1-\rho_2)&(1-\rho_2\rho_3)\\
&    &      & (1-\rho_3).
\end{matrix}$$
Considering all the 8 sign possibilities for $(\rho_1,\rho_2,\rho_3)$, one is able to show that $Q$ is bounded by 4 in every case: this fact is trivial when all the variables are positive, being $Q$ less than 1. For mixed signs, one gains information by either using the fact that $(1-\rho_i)(1-\rho_i\rho_j)(1-\rho_j)$ is less than 2 or by showing that the block of four factors $(1-\rho_1\rho_2)(1-\rho_1\rho_2\rho_3)(1-\rho_2)(1-\rho_2\rho_3)$ is less than 1, up to assuming some specific sign conditions on the $\rho_j$'s. The sharpest upper bound, equal to 4, is attained on the boundary, at the point given by $\rho_1=-1,\rho_2=0,\rho_3=-1$.

For values of $n$ up to 11, Pohst \cite{pohstRegulator} applied the same sketch of proof: for any sign condition on the $\rho_j$'s, one tries to estimate blocks of four or three factors by upper bounds which are 1 or 2 respectively, and the claim follows by checking every case. The result for any value of $n$ was obtained using a different approach (see \cite{battistonimolteniArxiv} for the details).
\end{proof}
\noindent
This is indeed a consistent improvement for the function $P$, and consequently for the Remak-Friedman inequality, whenever the considered numbers are real: this corresponds to a signature of the form $(n,0)$.
This better result was precisely the tool which allowed Astudillo, Diaz y Diaz and Friedman to classify the number fields with low regulator in the signatures $(8,0)$ and $(9,0)$.
\\ On the other side, one realizes that the classic estimate with upper bound $n^n$ is sharp for signatures which are totally complex or close to be totally complex. In order to understand this, we prove the following lemma, which is probably already known
in literature but for which we have not been able to find an
explicit reference.
\begin{lem}\label{Sharp}
Let $n\in\N$ be odd. Let $\zeta_n$ be a primitive $n$-th root of unity. Then 
$$P(1,\zeta_n,\zeta_n^2,\ldots,\zeta_n^{n-1})=n^n.$$
Let $n\in\N$ be even, and let $\zeta_n$ and $\zeta_{2n}$ be primitive roots of unity of order $n$ and $2n$ respectively. Then
$$P(1,\zeta_n,\zeta_n^2,\ldots,\zeta_n^{n-1})=n^n=P(\zeta_{2n},\zeta_{2n}^3,\ldots,\zeta_{2n}^{2n-1}).$$
\end{lem}

\begin{proof}
Assume that $n\in\N$ is odd. The powers $\zeta_n^j$ with $j\in\{0,1,\ldots,n-1\}$ are complex numbers with absolute value equal to 1; moreover, the value of every factor of $P$ is unchanged by multiplication with $|\zeta_n^j|$ for some suitable $j$ depending on the factor. Thus
$$P(1,\zeta_n,\zeta_n^2,\ldots,\zeta_n^{n-1}) = \prod_{0\leq i<j\leq n-1}\left|1-\frac{\zeta_n^i}{\zeta_n^j}\right|^2 =
\prod_{0\leq i<j\leq n-1}\left|\zeta_n^i-\zeta_n^j\right|^2$$
and the last term is equal to $|\text{disc}(x^n-1)|$, which is known to be $n^n$. The procedure and the result are exactly the same if one assumes that $n$ is even.\\
Now assume $n$ to be even and consider the function $P(\zeta_{2n},\zeta_{2n}^3,\ldots,\zeta_{2n}^{2n-1})$: being $\zeta_{2n}^{2j}=\zeta_n^j$ for every $j\in\{0,\ldots,n-1\}$, we have
$$P(\zeta_{2n},\zeta_{2n}^3,\ldots,\zeta_{2n}^{2n-1}) = \prod_{0\leq i<j\leq n-1}\left|1-\frac{\zeta_{2n}^{2i+1}}{\zeta_{2n}^{2j+1}}\right|^2 = \prod_{0\leq i<j\leq n-1}\left|1-\frac{\zeta_n^i}{\zeta_n^j}\right|^2 $$
and this value is equal to $n^n$ by the previous lines.
\end{proof}
\noindent
Lemma \ref{Sharp} shows that the classic estimate of $P$ with $n^n$ is sharp in the signature $(1, (n-1)/2)$ when $n$ is odd and in the signatures $(0,n/2)$ and $(2,(n-2)/2)$ when $n$ is even.

So, on one side we have recalled a much better estimate whenever the signature of the fields is $(n,0)$; on the other, we have seen that the classic estimate is sharp for signatures which are very near to correspond to the totally complex case. One could  wonder if for mixed signatures the sharp upper bounds are intermediate values between Pohst's bound $ 4^{\lfloor{n/2}\rfloor}$ and $n^n$; as Table \ref{TableMaximum} of Section \ref{section5} suggests, the fewer real embeddings one takes into account, the more these upper bounds seem to increase up to the classic estimate.

\section{Analytic setting of the problem}

\subsection{Definition and examples with low signatures}\label{subsection4.1}
As stated in the previous section, we want to study a possible improvement of the Remak-Friedman's inequality associated to the term defined in \eqref{FakeDiscriminant}. In the first part of this section we introduce an analytic setting in order to carry this study and we recover results for low signatures.

Let $n\in \N$ be an integer greater than $1$ and let $(r_1,r_2)$ be a pair of non-negative integers such that $n=r_1+2r_2$. Consider the set
\begin{align*}
 A_{n,r_2}\coloneqq \{&(\epsi_1,\ldots,\epsi_n)\in\C^n\colon 0 < |\epsi_1|\leq |\epsi_2|\leq \cdots\leq |\epsi_n|\leq 1,  r_1\text{ of the }\epsi_j\text{'s being real,} \\
 &\text{the remaining ones forming }r_2\text{ pairs of complex conjugated numbers}\}.  
\end{align*}
Define the function
\begin{equation*}
    Q(n,r_2,\cdot): A_{n,r_2}\rightarrow\R
\end{equation*}
\begin{equation}\label{FunctionQ}
    Q(n,r_2,(\epsi_1,\ldots,\epsi_n))\coloneqq \prod_{1\leq i<j\leq n}\left|1-\frac{\epsi_i}{\epsi_j}\right|.
\end{equation}
The square of $Q(n,r_2,(\epsi_1,\ldots,\epsi_n))$ is the quantity $P(\epsi_1,\ldots,\epsi_n)$ defined in (\ref{FakeDiscriminant}). We call the pair $(r_1,r_2)$ \textbf{the signature of the function} $Q(n,r_2,\cdot)$, in order to agree with the signature of number fields.\\
Given $Q(n,r_2,(\epsi_1,\ldots,\epsi_n))$ as above, assume that the pairs of complex conjugated numbers are $\{\epsi_{j_1},\epsi_{j_1+1}\},\{\epsi_{j_2},\epsi_{j_2+1}\},\ldots,\{\epsi_{j_{r_2}},\epsi_{j_{r_2+1}}\}$, with $j_i\in\{1,\ldots,n-1\}$. We define the following change of variables.
\begin{itemize}
    \item If $\epsi_i$ and $\epsi_{i+1}$ are both real, define $x_i\coloneqq \epsi_i/\epsi_{i+1}$.
    \item If $\epsi_{i}$ is real and $\epsi_{i+1}$ is complex, define $x_i\coloneqq\epsi_i/|\epsi_{i+1}|$.
    \item If $\epsi_i$ and $\epsi_{i+1}$ are complex conjugated, define $x_{i}\coloneqq\cos\theta$ with $\theta\coloneqq\arg \epsi_{i}$.
    \item If $\epsi_{i}$ is complex and $\epsi_{i+1}$ is real, define $x_i\coloneqq|\epsi_i|/\epsi_{i+1}$.
    \item If $\epsi_{i}$ and $\epsi_{i+1}$ are complex but not conjugated, define $x_i\coloneqq |\epsi_i|/|\epsi_{i+1}|$.
\end{itemize}
The function $Q$  can be replaced with several functions
\begin{equation*}\label{FunctionQChange}
    Q(n,r_2,\{j_1,\ldots,j_{r_2}\},\cdot): [-1,1]^{n-1}\rightarrow\R,
\end{equation*}
each one obtained from $Q$ by means of the above change of variables. We call the set of indexes $\{j_1,\ldots,j_{r_2}\}$ an admissible set of indexes: if $r_2=0$, we set $\{j_1,\ldots,j_{r_2}\}=\emptyset$. 

We define the number 
\begin{equation}\label{definitionM}
    M(n,r_2)\coloneqq  \max_{ \substack{\{j_1,\ldots,j_{r_2}\}\text{ admissible }\\(x_1,\ldots,x_{n-1})\in [-1,1]^{n-1}}} Q(n,r_2,\{j_1,\ldots,j_{r_2}\},(x_1,\ldots,x_n)).
\end{equation}


\begin{coro}
Let $K=\Q(\epsi)\in\Fr$ with $\epsi\in\Ok^*\setminus\mu_K$, and let $\epsi_1,\ldots,\epsi_n$ be its conjugates ordered in increasing absolute value. Then one can replace the term $n\log n$ in \eqref{D1} with $2\log M(n,r_2)$.
\end{coro}

\begin{coro}
By Theorem \ref{thmPohst}, one has $M(n,0)= 2^{\lfloor n/2 \rfloor}$ for every $n\geq 2$.
\end{coro}

\begin{coro}\label{CoroValuesM}
By Lemma \ref{Sharp} one has $M(n,(n-1)/2)= n^{n/2}$ for every odd integer $n\geq 3$, and $M(n,n/2) =  M(n,(n-2)/2)  = n^{n/2}$ for every even integer $n\geq 2$.
\end{coro}

\noindent
In particular, we know from Corollary \ref{CoroValuesM} that the maximum $M(n,r_2)$ of the function $Q$ cannot be improved for the signatures $(1,1), (2,1)$ and $(0,2)$. However, in the following lines we show how to recover the corresponding value of $M(n,r_2)$ with an approach different from the one used in Lemma \ref{Sharp}; instead, we will try to imitate Pohst's proof for the totally real signatures, using the proper change of variables and studying $Q(n,r_2,\cdot)$ analytically.
\begin{itemize}
    \item Consider first the signature $(1,1)$. Given $(\epsi_1,\epsi_2,\epsi_3)\in A_{3,1}$, let us assume that $\epsi_1$ is real and $|\epsi_2|=|\epsi_3|=1$ with $\epsi_2=\Bar{\epsi}_3$. Call  $x\coloneqq\epsi_1$ and $a\coloneqq\cos\theta$ with $\epsi_2=\exp(i\theta)$; then $(x,a)\in [-1,1]^2$ and the function $Q(3,1,\{2\},\cdot)$ is extended over $[-1,1]^2$ becoming
$$\begin{matrix}
& Q(3,1,\{2\},(x,a))\coloneqq& (1-2xa+x^2)\\
& & \hspace{0.2cm} 2\sqrt{1-a^2}
\end{matrix}$$
where the right hand side is assumed to be a product of all the written factors.\\
Then $Q(3,1,\{2\},(-x,-a))=Q(3,1,\{2\},(x,a))$ and the function is immediately seen to be maximized at the point $(1,-1/2)$ providing the value $3^{3/2}$. One notices that, thanks to the previous change of variables, this choice of $x$ and $a$ corresponds exactly to the third roots of unity which are known to give the correct value of $M(3,1)$ by Lemma \ref{Sharp}.\\
If one supposes instead that $\epsi_3$ is a real number and that $\epsi_1=\Bar{\epsi}_2$, then the boundary condition given by $A_{3,1}$ yields $\epsi_3\in\{\pm 1\}$ and we can take $\epsi_3=1$ without loss of generality (otherwise, one simply changes the sign to every $\epsi_j$). Being $\epsi_1= r\exp(i\theta)$ with $r\in[0,1]$ and defining $a\coloneqq\cos(\theta)\in [-1,1]$, the function $Q(3,1,\{1\},\cdot)$ can be extended again over $[-1,1]^2$ and becomes
$$\begin{matrix}
& Q(3,1,\{1\},(r,a))\coloneqq& (1-2ra+r^2)\cdot 2\sqrt{1-a^2}
\end{matrix}$$
which again is maximized at the point $(r,a)=(1,-1/2)$ corresponding to the third roots of unity and provides a maximum equal to $3^{3/2}$. Therefore, these two results put together give $M(3,1)=3^{3/2}$.
\vspace{0.1cm}
    \item Let us check now what happens for the signature $(2,1)$: for sake of simplicity, we only consider the case when $\epsi_1$ and $\epsi_2$ are real, while $\epsi_3$ and $\epsi_4$ are complex conjugated and of absolute value 1. Define then $x\coloneqq \epsi_1/\epsi_2$, $y\coloneqq\epsi_2$ and $a\coloneqq\cos(\theta)$ where $\epsi_3=\exp(i\theta)$; we have $(x,y,a)\in [-1,1]^3$ and again we can extend the function $Q(4,1,\cdot)$ over the hypercube obtaining the expression
$$\begin{matrix}
& Q(4,1,\{3\},(x,y,a))=& (1-x)&(1-2xya+(xy)^2)\\
& & & (1-2ya+y^2)\\
& & & 2\sqrt{1-a^2}.
\end{matrix}$$
We know that this function is estimated by $16$ from Theorem \ref{ThmBertin}: one verifies that this value is attained precisely at the point $(x,y,a)=(-1,1,0)$, which corresponds to the 4-th roots of unity via the change of variables. Studying the remaining cases given by the different choices of admissible indexes and the corresponding change of variables, one verifies that the maximum of 
$Q(4,1,\{j\},\cdot)$ with $j\in\{1,2\}$ is always equal to 16, so that $M(4,1)=16$.
\vspace{0.15cm}
\item For the signature $(0,2)$ we do not have to consider different subcases: in fact, we always have $\Bar{\epsi}_1=\epsi_2$ and $\Bar{\epsi}_3=\epsi_4$ and we can write $\epsi_1=r\exp(i\theta)$, $\epsi_3= s\exp(i\phi)$ with $0\leq r\leq s\leq 1$. Defining $x\coloneqq r/s, a\coloneqq \cos\theta$ and $b\coloneqq \cos\phi$, the function $Q(4,2,\{1,3\},\cdot)$ assumes the form
$$\begin{matrix}
& Q(4,2,\{1,3\},(x,a,b))=& &4\sqrt{1-a^2}\sqrt{1-b^2}\\
& & &((1+x^2)^2-4x(1+x^2)ab +4x^2(-1+a^2+b^2)).
\end{matrix}$$
We know that this function is bounded by $16$ from Theorem \ref{ThmBertin}, and this value is precisely attained at the point $(x,a,b)=(1,1/\sqrt{2},-1/\sqrt{2})$ which corresponds exactly to the numbers $\epsi_j = \zeta_8^{2j+1}$ via the change of variables. Thus, $M(4,2)=16$ as we expected.
\end{itemize}

\subsection{Studying the signature (3,1)}\label{sectionsignature31}
Signature $(3,1)$ is the first one for which we no longer have information due to Lemma \ref{Sharp}, so in this part of the section we try to apply the previous analytic setting in order to detect the correct value of $M(5,1)$.

Let us begin by assuming again that $\epsi_4$ and $\epsi_5$ are complex conjugated: with the change of variables $x\coloneqq\epsi_1/\epsi_2, y\coloneqq\epsi_2/\epsi_3, z\coloneqq\epsi_3, a\coloneqq \cos(\theta)$ where $\epsi_4=\exp(i\theta)$, the function $Q(5,1,\{4\},\cdot)$ can be extended over $[-1,1]^4$ and assumes the form
$$\begin{matrix}
& Q(5,1,\{4\},(x,y,z,a))=& (1-x)&(1-xy)&(1-2xyza+(xyz)^2)\\
& & &(1-y) &(1-2yza+(yz)^2)\\
& & &  &(1-2za+z^2)\\
& & &  & 2\sqrt{1-a^2}.
\end{matrix}$$
\begin{lem}\label{LemmaMaxQ}
The maximum of $Q(5,1,\{4\},(x,y,z,a))$ is attained at a point $(x,y,z,a)$ with $z=1$ and $a\neq \pm 1$.
\end{lem}
\begin{proof}
Observe that the function $Q$ is not negative over $[-1,1]^4$, is strictly positive in the interior $(-1,1)^4$ and is greater than 1 on part of the interior, so the maximums of $Q$ and of $\log Q$ are attained at the same point. Let us assume that the point of maximum is $(x,y,z,a)$ such that $(z,a)\in(-1,1)^2$. Then we have
$$
\begin{cases}
\partial_z\log Q = 0\\
\partial_a\log Q = 0
\end{cases}
$$
and this system of partial derivatives has the form
$$
\begin{cases}
\sum_{j=1}^3 \frac{-2\alpha_j a + 2\alpha_j^2 z}{1-2\alpha_j za + \alpha_j^2 z^2} = 0 & (\text{I}) \\
-\frac{a}{1-a^2} + \sum_{j=1}^3 \frac{-2\alpha_j z}{1-2\alpha_j z a + \alpha_j^2 z^2} = 0 &(\text{II})
\end{cases}
$$
where $\alpha_1\coloneqq xy$, $\alpha_2\coloneqq y$, $\alpha_3\coloneqq1$.\\
Now we manipulate the lines of the system in order to get
\begin{align*}
0 = z\cdot(\text{I}) - a\cdot(\text{II})
  &=
\sum_{j=1}^3 \frac{-2\alpha_j a z + 2\alpha_j^2 z^2}{1-2\alpha_j z a + \alpha_j^2 z^2}
+\frac{a^2}{1-a^2} + \sum_{j=1}^3 \frac{2\alpha_j a z}{1-2\alpha_j z a + \alpha_j^2 z^2}
\\
&=
\frac{a^2}{1-a^2} + \sum_{j=1}^3 \frac{2\alpha_j^2 z^2}{1-2\alpha_j z a  + \alpha_j^2 z^2}.
\end{align*}
Every term in the above sum is non-negative, and so we must have $a=0$ and $\alpha_j z =0$ for every $j$; being $\alpha_3=1$, it must be $z=0$. Thus the maximum is attained at a point $(x,y,z,a)$ which satisfies the condition $(z,a)=(0,0)$: but $Q(5,1,\{4\},(x,y,0,0))= 2 Q(3,0,\emptyset,(x,y))$, which by Theorem \ref{thmPohst} is bounded by $2\cdot 4=8$. This estimate is clearly in contradiction with the behaviour of $Q(5,1,\{4\},\cdot)$ because $Q(5,1,\{4\},(0,-1,1,0))= Q(4,1,\{3\},(-1,1,0))=16$; thus the maximum point must have the parameters
$(z,a)$ on the boundary of $[-1,1]^2$. \\
Now, the values $a=\pm 1$ force $Q$ to be 0, and so we are left with $z= \pm 1$ and
$a\neq \pm 1$. Being
\[
G(x,y,z,a) = G(x,y,-z,-a)
\]
we can finally assume $z=1$ and $a\neq \pm 1$.
\end{proof}
\noindent
Thanks to this lemma, the research of the maximum of $Q(5,1,\{4\},\cdot)$ is equivalent to studying the maximum of the function
$$\begin{matrix}
& Q(5,1,\{4\},(x,y,1,a))=& (1-x)&(1-xy)&(1-2xya+(xy)^2)\\
& & &(1-y) &(1-2ya+(y)^2)\\
& & & & 4(1-a)\sqrt{1-a^2}.
\end{matrix}$$

\begin{conj}\label{conj1}
    The maximum of $Q(5,1,\{4\},(x,y,1,a))$ is $16.6965\ldots$ and is attained at the point $(x,y,z,a) = (1/\sqrt{7},-1,1,1/(2\sqrt{7}))$.
\end{conj}

\begin{aprf}
Let us first define the function
$$\begin{matrix}
& R(x,y,1,a)=& (1-x)&(1-xy)&(1-2xya+(xy)^2)\\
& &  & (1-y) &(1-2ya+(y)^2)
\end{matrix}$$
which satisfies the relation $$Q(5,1,\{4\},(x,y,1,a)) = R(x,y,1,a)\cdot 4(1-a)\sqrt{1-a^2}.$$
This choice is done in order to study the partial derivatives with respect to $x$ and $y$ of $Q$ without carrying the factor which depends only on $a$.
Let $R_x(x,y,1,a)\coloneqq \partial_x R (x,y,1,a)$ and $R_y(x,y,1,a)\coloneqq \partial_y R (x,y,1,a)$.\\ 
The research of the maximum point of $Q(5,1,\{4\},\cdot)$ is carried by starting a numerical search on PARI/GP for the values that the function $Q$ assumes over specific sub-regions of $[-1,1]^3$, each search depending on a value of $a$ in a finite set.\\
In fact, let us vary the value of $a$ between $-0.999$ and $0.999$, with steps of size $1/1536$: for each one of these choices, we study the following quantities.
\begin{itemize}
    \item The maximum of $Q(5,1,\{4\},(x,-1,1,a))$ over $x\in [-1,1]$: in this case we have assumed the only meaningful boundary condition on $y$ (because the function is equal to zero if $y=1$) and we look for the maximum value by selecting numerically, via the command \textbf{polrootsreal()}, the real roots of the partial derivative $R_x(x,-1,1,a)$ such that $|x|\leq 1$, and we compute $Q(5,1,\{4\},\cdot)$ for these values of $x$.
    \item The maximum of $Q(5,1,\{4\},(-1,y,1,a))$ over $y\in [-1,1]$: this case is similar to previous one and this time we look for the real roots of $R_y(-1,y,1,a)$, evaluating then $Q(5,1,\{4\},\cdot)$ over the roots $y$ such that $|y|\leq 1$.
    \item The maximum of $Q(5,1,\{4\},(x,y,1,a))$ over the open set $\{(x,y)\in (-1,1)^2\}$: in this case we compute the common real roots of the polynomials $R_x(x,y,1,a)$ and $R_y(x,y,1,a)$ and the we evaluate the function $Q(5,1,\{4\},\cdot)$ over the roots $(x,y)$ with $|x|< 1$ and $|y|<1$. The numerical computation of the roots is done by studying the roots of the resultant of $R_x$ and $R_y$ with respect to the variable $x$: the needed PARI/GP command is \textbf{polresultant()}.
\end{itemize}
For every choice of $a$ in our interval, one looks for the regions where bigger values of $Q$ are obtained and notices that these values are attained at the boundary  $y=-1$; we can thus continue our study assuming this boundary condition.
\\
With this assumption the function $Q(5,1,\{4\},(x,-1,1,a))$ becomes now
$$Q(5,1,\{4\},(x,-1,1,a)) = 16(1-x^2)(1+2xa+x^2)(1-a^2)^{3/2}$$
and we look for its maximum, being finally able to give a precise analytic study. Surely $x$ must not be equal to $\pm 1$, otherwise the function is zero, and thus we must look for $x\in (-1,1)$: in order to study the partial derivative with respect to $x$, let us consider the factors which depend on $x$ by defining
$$S(x,a)\coloneqq (1-x^2)(1+2xa+x^2).$$
We have $S_x(x,a): = \frac{\partial S}{\partial_x} (x,a) = -4x^3 - 6 a x^2 + 2a $ and we study $S_x(x,a)=0$: this equation gives the condition
\begin{equation}\label{RelationG&X}
a = \frac{2x^3}{1-3x^2}    
\end{equation}
and one verifies that the function $a(x)$ defined by the rule above has positive derivative 
$$ a'(x)\coloneqq 6x^2(1-x^2)/(1-3x^2)^2$$ 
over $x\in [-1,1]^2$, which in turn implies that (\ref{RelationG&X}) gives a bijective correspondence: hence, for every value of $a\in (-1,1)$, there is a unique $x\coloneqq x(a)\in (-1,1)$ such that $S(x(a),a)=0$.

We finally study $16 S(x(a),a)(1-a^2)^{3/2}$: derive this in $a$ in order to get
$$\left[\frac{\partial S}{\partial x_1}(x(a),a)\frac{\partial x(a)}{\partial a} + \frac{\partial S}{\partial x_2}(x(a),a)\right](1-a^2)^{3/2} + S(x(a),a)(-3a)(1-a^2)^{1/2}=0.$$
By definition of $x(a)$ we have $\frac{\partial S}{\partial x_1}(x(a),a)=0$ and so
\begin{align*}
    \text{ }&\frac{\partial S}{\partial x_2}(x(a),a)(1-a^2) + S(x(a),a)(-3a)\\
    &= 2x(1-x^2)(1-a^2) - 3a (1-x^2)(1+2xa+x^2) = 0.
\end{align*}
Since the maximum not attained at $x=\pm 1$, we reduce ourselves to study
$$2x(1-a^2) - 3g (1+2xa+x^2) = 0.$$
Using (\ref{RelationG&X}) in this equation we finally get
$$\frac{-14x^7+30x^5-18x^3+2x}{9x^4-6x^2+1}=0 .$$
The numerator factorizes as $x(x^2-1)^2(7x^2-1)$, and the only zeros which do not annihilate the function $Q$ are attained at $x= \pm 1/\sqrt{7}$, implying $a= \pm 1/(2\sqrt{7})$ from (\ref{RelationG&X}): evaluating $Q(5,1,\{4\}, (1/\sqrt{7}, -1, 1, 1/(2\sqrt{7})))$ gives the maximum value $16.6965\ldots$
\end{aprf}

\begin{rmk}
The procedure shown above cannot be considered a proof because of the discrete process given by considering a finite, even if large, number of points instead of continuous intervals, and this is why we refer to the lines above as part of a conjecture.
\end{rmk}

\begin{rmk}
In the PARI/GP program used for the computation, some errors resulted while evaluating the resultant over $y=0$: however, this is not a real problem, since 
$$Q(5,1,\{4\},(x,0,1,a)) = (1-x)4(1-a)\sqrt{1-a^2}\leq 2\cdot 3\sqrt{3} = 6\sqrt{3} < 12$$ 
and meanwhile we know that $Q(5,1,\{4\},(0,-1,1,0))=Q(4,1,\{3\},(-1,1,0))=16$, so that surely an absolute maximum point for our original function cannot be attained over $y=0$.
\end{rmk}

\begin{rmk}
Even if the conjecture was actually proved, we could not conclude from this that $M(5,1) = 16.6965\ldots$ because we would have proved an estimate just for $Q(5,1,\{4\},\cdot)$, while nothing has been obtained yet for $Q(5,1,\{j\},\cdot)$ with $j\in\{1,2,3\}$: these seem to be functions for which it is not possible to reduce the number of variables from 4 to 3 just like in Lemma \ref{LemmaMaxQ}.
\end{rmk}

\section{Conjectures and applications}\label{section5}
\subsection{Experiments and conjectures on the maximums}
As described in the previous section, we are currently not able to prove that $M(5,1)=16.6965\ldots$ because we cannot make a good study of the different cases which cover all the possibilities for the functions $Q(5,1,\{j\},\cdot)$. Whenever we increase the degree and we change the signature, the situation becomes more and more complicated: the functions $Q(n,r_2,\{j_1,\ldots,j_{r_2}\},\cdot)$ are polynomials of several variables for which we are not able to prove anything in a rigorous way.\\
We decided then to content ourselves with a conjectural estimate of the maximum values $M(n,r_2)$ for $n\leq 8$ and see how these values would modify the study of minimum regulators given by Remak-Friedman's inequality. This was pursued by means of the following heuristics. 
\begin{itemize}
    \item First of all, for every signature $(r_1,r_2)$ and for every admissible ordering $\{j_1,\ldots,j_{r_2}\}$, we wrote the function $Q(n,r_2,\{j_1,\ldots,j_{r_2}\},(x_1,\ldots,x_{n-1}))$ in PARI/GP and we studied its evaluations over $10^5$ random $(n-1)$-ples $(x_1,\ldots,x_{n-1})\in [-1,1]^{n-1}\cap\Q^{n-1}$. After every evaluation, we confronted the obtained value with the biggest value obtained up to that iteration; once all the admissible orderings have been covered, we confronted all the maximum values we found and we saved the biggest value. We also kept track of the points in $[-1,1]^{n-1}\cap\Q^{n-1}$ at which the maximum values were attained.
    \vspace{0.1cm}
    \item Later, we wrote the functions $\Big( Q(n,r_2,\{j_1,\ldots,j_{r_2}\},\cdot)\Big)^2$ in MATLAB and we gave a numerical optimization for these functions on $[-1,1]^{n-1}$ via the MATLAB Optimization Toolbox. This was done using the \textbf{GlobalSearch} and \textbf{MultiStart} commands, which numerically detect global minimums, and we applied them on the function $-Q^2$: the command was iterated several times, each iteration depending on a randomly chosen starting point $(x_1,\ldots,x_{n-1})\in[-1,1]^{n-1}$ given as input for the optimization. As before, each iteration gave a candidate maximum and the points at which this maximum is attained: once all the iterations were done, we saved only the biggest maximum found for every signature and for every admissible ordering.
\end{itemize}
For every degree $n\leq 8$ and for every corresponding signature $(r_1,r_2)$, the two heuristic procedures provided the same values for $M(n,r_2)$ and the same points at which the values are attained. Moreover, once the degree and the signature were fixed, the same maximum value has been found for every admissible ordering $\{j_1,\ldots,j_{r_2}\}$. 

In Table \ref{TableMaximum} we present the conjectured upper bounds $M(n,r_2)$, and from that we form the following conjectures.
\begin{conj}
The values in Table \ref{TableMaximum} are the correct values of the upper bounds $M(n,r_2)$ for $n\leq 8$ and $r_2\leq 4$.    
\end{conj}
\begin{conj}\label{conjIncreasing}
For every $r_2\in\N$ there exists $C(r_2)\in\N$ such that $M(n+2,r_2)=2M(n,r_2)$ for every $n\geq C(n,r_2)$.
\end{conj}
\begin{conj}\label{conjr2=1}
$C(1)=4$ and for every $n\geq C(1)$, after a suitable change of variables, the maximum value $M(n,1)$ is attained at $(x_1,\ldots x_{n-3},x_{n-2},a) = (-1,0,-1,0,\ldots,-1,0)$ for $n$ even and $(x_1,\ldots x_{n-3},x_{n-2},a) = (-1/\sqrt{7}, -1,-1/\sqrt{7},-1,\ldots,-1/\sqrt{7},-1,1,1/(2\sqrt{7})). $
\end{conj}
\begin{table}[H]
    \centering
    \begin{tabular}{l|c|c|c|c|c|c|c}
        \backslashbox{$r_2$}{$n$} & 2 &3 & 4 & 5 & 6 & 7 & 8\\
         \hline
        0 & 2 & 2 & 4 & 4 & 8 & 8 & 16 \\
        \hline
        1 & & $3^{3/2}$ & 16 & $16.6965\ldots$ & 32 & $2\cdot 16.6965\ldots$ & 64 \\
        \hline
        2 & & & 16 & $5^{5/2}$ & $6^{6/2}$ & $245.8193\ldots$ & $7^{7/2}$\\
        \hline
        3 & & & & &$6^{6/2}$ & $7^{7/2}$ & $8^{8/2}$   \\
        \hline
        4 & & & & & & & $8^{8/2}$
    \end{tabular}
    \caption{ The conjectured values for $M(n,r_2)$.}
    \label{TableMaximum}
\end{table}
Conjecture \ref{conjIncreasing} is suggested by the first two lines of Table \ref{TableMaximum} and in particular from Theorem \ref{thmPohst} for $r_2=0$: in fact, this yields $C(0)=2$ and for every $n\geq C(0)$, after a suitable change of variables, the maximum value $M(n,0)$ is attained at the point $(x_1,\cdots,x_n)=(0,-1,0,-1,\ldots,0,-1)$ for $n$ even and $(-1,0,-1,\ldots,0,-1)$ for n odd.

Conjecture \ref{conjr2=1} is suggested by the second line of Table \ref{TableMaximum} and from the discussion presented in Section \ref{sectionsignature31}: in fact, this is a generalization of Conjecture \ref{conj1}. The points at which the maximums are conjectured to be attained have been revealed by the numerical optimization process described at the beginning of this section.
\subsection{Application to the minimum regulator problem}\label{SectionFinale}
In this last section, we prove Theorems \ref{ThmDeg8}, \ref{ThmDeg5} and \ref{ThmDeg7}.

\begin{proof}[Proof of Theorem \ref{ThmDeg8}]
Put $R_0=7.48$: then, replacing the term $8\log 8$ in \eqref{D1} with $2\log M(8,6)= 12\log 2$, the upper bound $D_1$ lowers from $43.7698$ to $35.6632$, and  $4g_{6,1}(\exp(-35.6632)) = 102.264\ldots > 7.48$. Being $4g_{6,1}(1/79259702) = 7.48749\ldots$, we conclude that a number field with signature $(6,1)$ and $R_K\leq R_0$ must have $|d_K|\leq 79259702$ and so, if it exists, is contained in a completely known list. Computation of actual values of the regulators like in Remark \ref{RemarkGRH} show that these fields are exactly the one described in the statement of the Theorem.
\end{proof}

\begin{rmk}
The signature $(6,1)$ is the only signature in degree 8 for which one can obtain results from the conjectural estimates. Surely signature $(2,3)$ is not affected because we know that $M(8,3)=8^{8/2}$ by Lemma \ref{Sharp}, so that in this case we are still stuck with the previous estimate given by Remak-Friedman's inequality.\\
For what concerns signature $(4,2)$, we would have an improvement given by the (conjectural) correct value $2\log M(8,2) = 2\log\lp7^{7/2}\rp = 7\log 7$ instead of the upper bound $8\log 8$. Unfortunately, using again $R_0=2.298$ as in Section \ref{sectionfirstattempt}, the new estimate would imply that number fields with signature $(4,2)$ and $R_K\leq R_0$ must have $|d_K|\leq \exp(35.3463)$, and $4g_{4,2}(\exp(-35.3463))$\\ 
$= -166.2009\ldots$; not even the adaptations described in \cite{regulators} work in this case.\\
At the moment, the only thing we can conclude for signature $(4,2)$ is that a field $K\in \mathcal{F}_{4,2}$ with $R_K\leq 2.298$ must be either the field generated by the polynomial $x^8 - x^6 - 6x^5 + 3x^3 + x^2 + 2x - 1$, having discriminant $15243125$ and regulator 2.2977\ldots or some possible field $K$ with $|d_K|\in (20829049, \exp(35.3463)).$
\end{rmk}

\begin{proof}[Proof of Theorem \ref{ThmDeg5}]
Put $R_0=1.73$. If we replace the factor $5\log 5$ in \eqref{D1} with $2\log M(5,1) = 2\log \lp16.6965\rp\ldots$, the new upper bound is $|d_K|\leq \exp(16.8961)$ and one has \\
$2g_{3,1}(\exp(-16.8961)) = 3.404\ldots$ \\
Since $2g(1/(\delta_{9,3})^5) = 2.158\ldots$ from \cite[Theorem 8]{regulators}, we  look for smaller upper bounds using the factor 4 in the computations. We have in fact $4g_{3,1}(1/48000) = 2.157\ldots$, and thus we get that any field of signature $(3,1)$ with $R_K\leq 2.15$ must have $|d_K|\leq 48000$. Studying the list containing these fields (obtained from the Kl\"{u}ners-Malle Database and LMFDB database), we get the desired result, since every such field satisfies condition \eqref{CorrettoRegolatoreCaso2} and we know the true values of their regulators.
\end{proof}


\begin{proof}[Proof of Theorem \ref{ThmDeg7}]
Put $R_0=8$: the geometric bound given by Remak-Friedman inequality would give the value $\exp(37.0334) > \exp(20.1)=: (\delta_{15,3})^7$, which would not be useful because $2 g_{5,1}(\exp(-37.0334))= -527.6403\ldots$. By replacing the factor $7\log 7$ with $2\log M(7,1)=2\log (2\cdot 16.6965\ldots)$, we obtain instead the upper bound $|d_K|\leq \exp(30.4288)$ which is way better because $2 g_{5,1}(\exp(-30.4288)) = 10.2565\ldots > R_0$.\\
Now, $2g_{5,1}(1/(\delta_{15,3})^7) = 13.705\ldots$ and so we can use the factor 4: one verifies that $4g_{5,1}(1/(2\cdot 10^7)) = 8.1578\ldots$, and so we must look at the list of 528 number fields of signature $(5,1)$ with $|d_K|\leq 2\cdot 10^7$ (the lists are collected from the same sources as before). Just as for the previous remark, the values of the regulators of these fields given by PARI/GP are correct, because  \eqref{CorrettoRegolatoreCaso2} is satisfied.
\end{proof}


\begin{thebibliography}{10}

\bibitem{regulators}
S.~Astudillo, F.~Diaz~y Diaz, and E.~Friedman.
\newblock Sharp lower bounds for regulators of small-degree number fields.
\newblock {\em J. Number Theory}, 167:232--258, 2016.

\bibitem{battistoniMinimum}
F.~Battistoni.
\newblock The minimum discriminant of number fields of degree 8 and signature
  {$(2,3)$}.
\newblock {\em J. Number Theory}, 198:386--395, 2019.

\bibitem{battistoni2019Arxiv}
F.~Battistoni.
\newblock On small discriminants of number fields of degree 8 and 9.
\newblock available at \url{https://arxiv.org/pdf/1910.07208.pdf}, 2019.

\bibitem{battistonimolteniArxiv}
F.~Battistoni and G.~Molteni.
\newblock An elementary proof for a generalization of a pohst's inequality.
\newblock {\em Arxiv preprint}, 2021.
\newblock available at \url{https://arxiv.org/pdf/2101.06163.pdf}.

\bibitem{bertin}
M.~J. Bertin.
\newblock Sur une conjecture de {P}ohst.
\newblock {\em Acta Arith.}, 74(4):347--349, 1996.

\bibitem{casselsGeometryNumbers}
J.~W.~S. Cassels.
\newblock {\em An introduction to the geometry of numbers}.
\newblock Classics in Mathematics. Springer-Verlag, Berlin, 1997.
\newblock Corrected reprint of the 1971 edition.

\bibitem{cohenComputational}
H.~Cohen.
\newblock {\em A course in computational algebraic number theory}, volume 138
  of {\em Graduate Texts in Mathematics}.
\newblock Springer-Verlag, Berlin, 1993.

\bibitem{y1980tables}
F.~Diaz~y Diaz.
\newblock {\em Tables minorant la racine {$n$}-i\`eme du discriminant d'un
  corps de degr\'e {$n$}}, volume~6 of {\em Publications Math\'ematiques
  d'Orsay 80 [Mathematical Publications of Orsay 80]}.
\newblock Universit\'e de Paris-Sud, D\'epartement de Math\'ematique, Orsay,
  1980.

\bibitem{friedmanAnalyticRegulator}
E.~Friedman.
\newblock Analytic formulas for the regulator of a number field.
\newblock {\em Invent. Math.}, 98(3):599--622, 1989.

\bibitem{friedmanTotalPositivity}
E.~Friedman.
\newblock Regulators and total positivity.
\newblock {\em Publ. Mat.}, (Proceedings of the Primeras Jornadas de Teor\'{i}a
  de N\'{u}meros):119--130, 2007.

\bibitem{RamirezRaposo}
E.~Friedman and G.~Ramirez-Raposo.
\newblock Filling the gap in the table of smallest regulators up to degree 7.
\newblock {\em J. Number Theory}, 198:381--385, 2019.

\bibitem{heckeLectures}
E.~Hecke.
\newblock {\em Lectures on the theory of algebraic numbers}, volume~77 of {\em
  Graduate Texts in Mathematics}.
\newblock Springer-Verlag, New York-Berlin, 1981.
\newblock Translated from the German by George U. Brauer, Jay R. Goldman and R.
  Kotzen.

\bibitem{karlinTotalPositivity}
S.~Karlin.
\newblock {\em Total positivity. {V}ol. {I}}.
\newblock Stanford University Press, Stanford, Calif, 1968.

\bibitem{klunersMalle}
J.~Kl\"{u}ners and G.~Malle.
\newblock A database for number fields.
\newblock Available at \url{http://galoisdb.math.upb.de/home}.

\bibitem{matlabOptimization}
Matlab optimization toolbox, 2018 B.
\newblock The MathWorks, Natick, MA, USA.

\bibitem{pari}
{PARI~Group}, Univ. Bordeaux.
\newblock {\em {PARI/GP version {\tt 2.12.0}}}, 2020.
\newblock available at \url{http://pari.math.u-bordeaux.fr/}.

\bibitem{pohstRegulator}
M.~Pohst.
\newblock Regulatorabsch\"{a}tzungen f\"{u}r total reelle algebraische
  {Z}ahlk\"{o}rper.
\newblock {\em J. Number Theory}, 9(4):459--492, 1977.

\bibitem{pohst1997algorithmic}
M.~Pohst and H.~Zassenhaus.
\newblock {\em Algorithmic algebraic number theory}, volume~30 of {\em
  Encyclopedia of Mathematics and its Applications}.
\newblock Cambridge University Press, Cambridge, 1997.
\newblock Revised reprint of the 1989 original.

\bibitem{pohst1975berechnung}
Michael Pohst.
\newblock Berechnung kleiner diskriminanten total reeller algebraischer
  zahlk{\"o}rper.
\newblock {\em Journal f{\"u}r die reine und angewandte Mathematik},
  1975(278-279):278--300, 1975.

\bibitem{remak}
R.~Remak.
\newblock \"{U}ber {G}r\"{o}ssenbeziehungen zwischen {D}iskriminante und
  {R}egulator eines algebraischen {Z}ahlk\"{o}rpers.
\newblock {\em Compositio Math.}, 10:245--285, 1952.

\bibitem{lmfdb}
{The LMFDB Collaboration}.
\newblock The {L}-functions and modular forms database.
\newblock available at \url{http://www.lmfdb.org}, 2013.

\end{thebibliography}
\end{document}